\documentclass[12pt,reqno]{amsart}
\usepackage{amssymb}
\usepackage{amsfonts}
\usepackage{amssymb,latexsym}
\usepackage{enumerate}
\usepackage{url}
\usepackage{mathrsfs}
\usepackage{bbm}
\makeatletter
\@namedef{subjclassname@2010}{%
  \textup{2010} Mathematics Subject Classification}
\makeatother

  [2002/01/19 v2.2g %
    AMS font definitions%
  ]
\DeclareFontFamily{U}{euf}{}
\DeclareFontShape{U}{euf}{m}{n}{%
  <5><6><7><8><9>gen*eufm%
  <10><10.95><12><14.4><17.28><20.74><24.88>eufm10%
  }{}
\DeclareFontShape{U}{euf}{b}{n}{%
  <5><6><7><8><9>gen*eufb%
  <10><10.95><12><14.4><17.28><20.74><24.88>eufb10%
  }{}

  [2002/01/19 v2.2g %
    AMS font definitions%
  ]
\DeclareFontFamily{U}{msb}{}
\DeclareFontShape{U}{msb}{m}{n}{%
  <5><6><7><8><9>gen*msbm%
  <10><10.95><12><14.4><17.28><20.74><24.88>msbm10%
  }{}

  [2002/01/19 v2.2g %
    AMS font definitions%
  ]
\DeclareFontFamily{U}{msa}{}
\DeclareFontShape{U}{msa}{m}{n}{%
  <5><6><7><8><9>gen*msam%
  <10><10.95><12><14.4><17.28><20.74><24.88>msam10%
  }{}

\newtheorem{theorem}{Theorem}[section]
\newtheorem{lemma}[theorem]{Lemma}

\theoremstyle{definition}

\numberwithin{equation}{section} 

\def\C{\mathcal C}

\begin{document}

\title[Approximations by multiple cosine and sine functions]
{Approximations by special values of multiple cosine and sine functions}

\author{Su Hu}
\address{Department of Mathematics, South China University of Technology, Guangzhou, Guangdong 510640, China}
\email{mahusu@scut.edu.cn}

\author{Min-Soo Kim}
\address{Department of Mathematics Education, Kyungnam University, Changwon, Gyeongnam 51767, Republic of Korea}
\email{mskim@kyungnam.ac.kr}


\begin{abstract}
Kurokawa and Koyama's multiple cosine function $\C_r(x)$ and Kurokawa's multiple sine function $S_{r}(x)$
are   generalizations of the classical cosine and sine functions from their infinite product representations, respectively.
For any fixed $x\in[0,\frac{1}{2})$, let  
 $$B=\left\{\frac{\log\C_r(x)}{\pi}~~\bigg|~~r=1,2,3,\ldots\right\}$$
 and $$C=\left\{\frac{\log S_r(x)}{\pi}~~\bigg|~~r=1,2,3,\ldots\right\}$$
 be the sets of special values of $\C_r(x)$ and $S_{r}(x)$ at $x$, respectively.
 
 In this paper, we will show that the real numbers can be strongly approximated by
linear combinations of elements in $B$ and $C$ respectively,  with rational coefficients. Furthermore, let
$$D=\left\{\frac{\zeta_{E}(3)}{\pi^2},\frac{\zeta_{E}(5)}{\pi^4}, \ldots, \frac{\zeta_{E}(2k+1)}{\pi^{2k}},\ldots; \frac{\beta(4)}{\pi^3},\frac{\beta(6)}{\pi^5}, \ldots, \frac{\beta(2k+2)}{\pi^{2k+1}},\ldots\right\}$$
be the set of special values of Dirichlet's eta and beta functions. We will prove that the set $D$ has a similar approximation property, where the coefficients are values of the derivatives of rational polynomials. 
Our approaches are inspired by recent works of Alkan \cite{Alkan} and Lupu-Wu \cite{LW}  as applications of the trigonometric integrals.
 \end{abstract}

\maketitle

\section{Introduction}
The Riemann zeta function is defined by
\begin{equation}~\label{Ri-zeta}
\zeta(s)=\sum_{n=1}^{\infty}\frac{1}{n^{s}},
\end{equation}
for Re$(s)>1$.
Let $$A=\{\zeta(3),\zeta(5), \ldots, \zeta(2k+1),\ldots\}$$ be the set of odd zeta values.
For $x\in\mathbb{R}$, let $\lfloor x\rfloor =\max\{m\in \mathbb {Z} \mid m\leq x\}$ and $\lceil x\rceil =\min\{m\in \mathbb {Z} \mid m\geq x\}$. 
We have the following formula connecting the trigonometric integral  $\int_{0}^{2\pi}t^{r}\log\left(2\sin\frac{t}{2}\right)dt$ and the elements in $A$:
\begin{equation}\label{(1)}
\int_{0}^{2\pi}t^{r}\log\left(2\sin\frac{t}{2}\right)dt=\sum_{k=1}^{\lfloor r/2\rfloor}(-1)^{k}\frac{r!(2\pi)^{r-2k+1}}{(r-2k+1)!}\zeta(2k+1),
\end{equation}
where $r=2,3,4,\ldots$ (see \cite[p. 3748, (2.9)]{Alkan}).
From the integral (\ref{(1)}), Alkan \cite[Theorem 1]{Alkan} proved that the real numbers can be strongly approximated by  linear combinations
of elements in $A$ with coefficients in $\mathbb{Q}[\pi^{s}:1\leq s\leq r-2k+1]$, where $\mathbb{Q}[\pi^{s}:1\leq s\leq n]$ is
the $\mathbb{Q}$-vector space with basis $\pi^{s}$ for $1\leq s\leq n$. His result is an analogy of the classical Diophantine approximation of
Liouville numbers by rational numbers, that is, a real number $\alpha$ is named a Liouville number, if for every positive integer $n$, there exists a pair of integers $(m,n)$ with $n>1$ such that
$$0<\left|\alpha-\frac{m}{n}\right|<\frac{1}{n^{q}}.$$ 
It is also a reminiscent of a recent result by Brown and Zudilin \cite{BZ}, which
constructed infinitely many \textit{effective} rational approximations $\frac{m}{n}$ to $\zeta(5)$ satisfying
$$0<\left|\zeta(5)-\frac{m}{n}\right|<\frac{1}{n^{0.86}}.$$

Let $$A^{\prime}=\left\{\frac{\zeta(3)}{\pi^3}, \frac{\zeta(5)}{\pi^5}, \ldots, \frac{\zeta(2k+1)}{\pi^{2k+1}},\ldots\right\}.$$
Recently, applying the following trigonometric integral by Orr (\cite{Orr}, also see \cite[Lemma 2.1]{LW}):
\begin{equation}\label{(2)}
\begin{aligned}
\int_{0}^{2\pi}t^{r}\cot(t)dt&=\left(\frac{\pi}{2}\right)^{r}\left(\log 2+\sum_{k=1}^{\lfloor r/2\rfloor}\frac{r!(-1)^{k}(4^k-1)}{(r-2k)!(2\pi)^{2k}}\zeta(2k+1)\right)\\
&\quad+\delta_{\lfloor r/2\rfloor,r/2}\frac{r!(-1)^{r/2}\zeta(r+1)}{\pi^{r}}
\end{aligned}
\end{equation}
for  $r=1,2,3,\ldots$,
 Lupu and Wu \cite{LW} improved Alkan's result to obtain the following approximation, 
which shows that  the real numbers can be strongly approximated by  linear combinations of elements in $A^{\prime}$, where the coefficients are values of the derivatives of rational polynomials.

\begin{theorem}[{\cite[Theorem 1.1]{LW}}]\label{Theorem 4}
Let $k_{0}$ and $q$ be positive integers, let $\alpha$ be a real number. For any given integer $n\geq 3$, there exists an integer $r\leq n$ and a rational polynomial $P_{n}(t)=t^{2k_{0}}(1-t)^{2k_{0}}s_{r}(t)$
with degree $4k_{0}+r$ depending on $\alpha,n,q$ and $k_{0}$,
satisfying 
$$\left|\alpha-\sum_{k=k_{0}}^{2k_{0}+\lfloor \frac{r}{2}\rfloor}c_{k}\frac{\zeta(2k+1)}{\pi^{2k+1}}\right|\ll_{\alpha,q, k_{0}}\frac{1}{n^q}$$
when $r\geq 2$, where $$c_{k}=(-1)^{k}2\left[P_{n}^{(2k)}(1)\left(1-\frac{1}{4^{k}}\right)+P_{n}^{(2k)}(0)\right]$$
for all $k_{0}\leq k \leq 2k_{0}+\lfloor \frac{r}{2}\rfloor$.
\end{theorem}
 
In fact, there is a long history for investigating the connections between zeta values and the trigonometric integrals such as (\ref{(1)}) and (\ref{(2)}).
Let 
\begin{equation}\label{lam}
\begin{aligned}
\lambda(s)&=\sum_{n=0}^\infty\frac1{(2n+1)^s}, ~~~~\rm{Re}(s)>1
\end{aligned}
\end{equation}
be Dirichlet's lambda function (see \cite[p. 954, (1.9)]{HK2019}), which was studied by Euler under the notation $N(s)$ 
(see \cite[p. 70]{Var}).
Long time ago, Euler has obtained a formula of $\lambda(3)$:
\begin{equation}\label{(3)}
\lambda(3)=1+\frac{1}{3^{3}}+\frac{1}{5^3}+\cdots=\frac{\pi^2}{\log 2}+2\int_0^{\frac{\pi}{2}}t\log(\sin t) dt
\end{equation}
(see \cite[p. 63]{Var}),
which connects to the problem of finding an explicit formula for $\zeta(3)$ (see \cite[section 1.2]{HK2025} for a historical overview).
Generalizing (\ref{(3)}), for any $0\leq x<\pi$ and  $r=2,3,4,\ldots$, Koyama and Kurokawa \cite[Theorem 1]{KK05} showed that
\begin{equation}\label{(4)}
\int_0^x t^{r-2}\log(\sin t)dt=\frac{x^{r-1}}{r-1}\log\left(\sin x\right)
-\frac{\pi^{r-1}}{r-1}\log \mathcal S_r\left(\frac x \pi\right).
\end{equation}
Here $\mathcal S_r(x)$ is Kurokawa's multiple sine function of order $r=2,3,4,\ldots$ (see \cite{Ku1, Ku2, Ku3}):
\begin{equation}
\mathcal S_r(x)=\exp\left(\frac{x^{r-1}}{r-1}\right)\prod_{n=1}^\infty\left\{P_r\left(\frac xn\right)P_r\left(-\frac xn\right)^{(-1)^{r-1}}\right\}^{n^{r-1}},
\end{equation}
where 
\begin{equation}\label{Pr}
P_r(x)=(1-x)\exp\left(x+\frac{x^2}{2}+\cdots+\frac{x^r}{r}\right).
\end{equation}
By letting $r=1$, we recover the infinite product representation of the classical sine function:
$$\mathcal S_1(x)=2\sin(\pi x)=2\pi x \prod_{n=1}^\infty\left(1-\frac{x^2}{n^2}\right)$$
and by letting $r=2$, we get H\"older's double sine function $\mathcal S_2(x)$ (see \cite{Holder}):
\begin{equation}
\mathcal S_2(x)=e^x\prod_{n=1}^\infty\left\{\left(\frac{1-\frac xn}{1+\frac xn}\right)^ne^{2x}\right\}.
\end{equation}

Recently, we have proved a cosine counterpart integral of (\ref{(4)}):
\begin{equation}
\int_0^x t^{r-2}\log\left(\cos\frac{t}{2}\right)dt=\frac{x^{r-1}}{r-1}\log\left(\cos\frac x2\right)
-\frac{(2\pi)^{r-1}}{r-1}\log\C_r\left(\frac x{2\pi}\right)
\end{equation}
(see \cite[Theorem 2.1]{HK2025}).
Here $\C_r\left(x\right)$ is Kurokawa and Koyama's multiple cosine function or order $r=2,3,4,\ldots$:
\begin{equation}\label{mcos-def}
\begin{aligned}
\C_r(x)&=\prod_{n=-\infty, n\text{:odd}}^\infty P_r\left(\frac{x}{\frac n2}\right)^{(\frac n2)^{r-1}} \\
&=\prod_{n=1, n\text{:odd}}^\infty \left\{P_r\left(\frac{x}{\frac n2}\right)P_r\left(-\frac{x}{\frac n2}\right)^{(-1)^{r-1}}\right\}^{(\frac n2)^{r-1}}.
\end{aligned}
\end{equation}
 (see  \cite{KW03}, \cite{KW} and \cite{KW04}).
By letting $r=1$, we recover the infinite product representation of the classical cosine function:
\begin{equation}\label{cos1-ex}
\C_1(x)=2\cos(\pi x) \\
=2\prod_{n=1, n\text{:odd}}^\infty \left(1-\frac{x^2}{(\frac n2)^2} \right).
\end{equation}
And by letting $r=2,3,4$, we get the following examples, respectively:
\begin{equation}\label{mcos-ex}
\begin{aligned}
\C_2(x)&=\prod_{n=1, n\text{:odd}}^\infty \left\{\left(\frac{1-\frac x{(\frac n2)}}{1+\frac x{(\frac n2)}}\right)^{\frac n2}e^{2x}\right\}, 
\\
\C_3(x)&=\prod_{n=1, n\text{:odd}}^\infty \left\{\left(1-\frac{x^2}{(\frac n2)^2}\right)^{(\frac n2)^2}e^{x^2}\right\},
\\
\C_4(x)&=\prod_{n=1, n\text{:odd}}^\infty \left\{\left(\frac{1-\frac x{(\frac n2)}}{1+\frac x{(\frac n2)}}\right)^{(\frac n2)^3}e^{\frac{n^2}2 x+\frac23 x^3}\right\}
\end{aligned}
\end{equation}
(see \cite{KK,KW03,KW04}).

Furthermore, by applying the integrals 
$$\int_0^x t^{r-2}\log\left(\sin t\right)dt \quad\textrm{and}\quad\int_0^x t^{r-2}\log\left(\cos\frac{t}{2}\right)dt,$$
we can show that the special values of multiple sine and cosine functions will be expressed by the zeta values.
For example, \cite[Theorem 2.4]{HK2025} gives 
\begin{equation}\label{(1.14)}
\begin{aligned}
\log\C_r\left(\frac14\right)
&=\frac{\log2}{2^{2r-1}}-\frac{(r-1)!}{(2\pi)^{r-1}}\sin\left(\frac{r\pi}{2}\right)\zeta_E(r) \\
&\quad-\frac{r-1}{2^{2(r-1)}}\sum_{k=0}^{\left\lfloor \frac{r-2}{2}\right\rfloor }(-1)^k(2k)!\binom{r-2}{2k}\left(\frac2\pi\right)^{2k+1}
 \\
&\quad\times\beta(2k+2) \\
&\quad-\frac{r-1}{2^{2r-1}}\sum_{k=1}^{\left\lceil \frac{r-2}2\right\rceil}\frac{(-1)^{k-1}(2k-1)!}{\pi^{2k}}\binom{r-2}{2k-1}
\\
&\quad\times\zeta_E(2k+1)
\end{aligned}
\end{equation}
for $r=2,3,4,\ldots$.
Here  
\begin{equation}\label{A-zeta-1}
\zeta_E(s)=\sum_{n=1}^\infty\frac{(-1)^{n+1}}{n^{s}}, \quad\rm{Re}(s)>0
\end{equation}
 is the alternating zeta function
(also known as Dirichlet's eta or Euler’s eta function) and 
\begin{equation}\label{beta-def}
\beta(s)=\sum_{n=0}^\infty\frac{(-1)^{n}}{(2n+1)^s}, \quad\rm{Re}(s)>0
\end{equation} 
 is the alternating form of Dirichlet's lambda function (\ref{lam}) (also known as Dirichlet's beta function), which was studied by Euler under the notation  $L(s)$ (see \cite[p.~70]{Var}). 
 Let 
 $$
G=\sum_{n=0}^\infty\frac{(-1)^n}{(2n+1)^2}=0.915965594177219015\cdots
$$
be the Catalan constant.
As an example, if setting $r=3$ in (\ref{(1.14)}), then we get
 $$\log\C_3\left(\frac14\right)=\frac{\log2}{32}-\frac{G}{4\pi}+\frac{7\zeta_E(3)}{16\pi^2},$$
 which is equivalent to a formula of $\zeta(3)$:
 $$
\zeta(3)=\frac{4\pi^2}{21}\log\left(\frac{e^{\frac{4G}{\pi}}\C_3\left(\frac14\right)^{16}}{\sqrt2}\right)
$$ 
(see \cite[Corollary 2.7]{HK2025}).
 
Now for any $0\leq x<\frac{1}{2}$, let  
 $$B=\left\{\frac{\log\C_r(x)}{\pi}~~\bigg|~~r=1,2,3,\ldots\right\}$$
 and $$C=\left\{\frac{\log S_r(x)}{\pi}~~\bigg|~~r=1,2,3,\ldots\right\}$$
 be the sets of special values of multiple cosine and sine functions at $x$, respectively.
In this paper, by applying the  trigonometric integrals (see Lemmas \ref{2.1} and \ref{4.1})
$$\int_0^x t^{r}\tan(\pi t)dt\quad\textrm{and}\quad\int_0^x t^{r}\cot\left(\frac{\pi t}{2}\right)dt,$$
we show that for any fixed $x\in[0,\frac{1}{2})$, the real numbers can be strongly approximated by
linear combinations of elements in $B$ and $C$ respectively,  with rational coefficients 
(see Sections 2 and 4).
Furthermore, by considering the integral  $$\int_0^1 t^{r-2}\log\left(\cos\frac{\pi t}{4}\right)dt$$
(see Lemma \ref{3.1}), we prove that the set of zeta and beta values
$$D=\left\{\frac{\zeta_{E}(3)}{\pi^2},\frac{\zeta_{E}(5)}{\pi^4}, \ldots, \frac{\zeta_{E}(2k+1)}{\pi^{2k}},\ldots; \frac{\beta(4)}{\pi^3},\frac{\beta(6)}{\pi^5}, \ldots, \frac{\beta(2k+2)}{\pi^{2k+1}},\ldots\right\}$$
has a similar approximation property, where the coefficients are values of the derivatives of rational polynomials (see Section 3).

\section{The integral $\int_0^x t^{r}\tan(\pi t)dt$} 
For any fixed $x\in[0,\frac{1}{2})$, let $$B=\left\{\frac{\log\C_r(x)}{\pi}~~\bigg|~~r=1,2,3,\ldots\right\}.$$
In this section, we will show that  the real numbers can be approximated by
linear combinations of elements in $B$  with rational coefficients.
First, we need the following lemmas.
\begin{lemma}\label{2.1}
For $0\leq x<\frac{1}{2}$ and $r=1,2,3,\ldots$, we have
\begin{equation}
\int_{0}^{x}t^{r}\tan(\pi t)dt=-\frac{\log\C_{r+1}(x)}{\pi}.
\end{equation}
\end{lemma}
\begin{proof}
By \cite[Proposition 3.3]{HK2025}, for $0\leq x<1$ and $r=2,3,4,\ldots$, we have
\begin{equation}
\begin{aligned}
\log\C_r\left(\frac x2\right)&=-\frac{1}{2^r}\int_0^x \pi t^{r-1}\tan\left(\frac{\pi t}{2}\right)dt\\
&=-\pi \int_0^x  \left(\frac{t}{2}\right)^{r-1}\tan\left(\frac{\pi t}{2}\right)d\left(\frac{t}{2}\right)\\
&=-\pi\int_{0}^{\frac{x}{2}}{t}^{r-1}\tan(\pi t)dt,
\end{aligned}
\end{equation}
which is equivalent to
$$\int_{0}^{x}t^{r}\tan(\pi t)dt=-\frac{\log\C_{r+1}(x)}{\pi}$$
for $0\leq x<\frac{1}{2}$ and $r=1,2,3,\ldots$.
\end{proof}

\begin{lemma}\label{2.2}
For any $x\in[0,\frac{1}{2})$ and a positive integer $k_{0}$. Let $\alpha$ be a real number. 
Then there exists a function $f_{\alpha}\in C^{\infty}[0,x]$ such that
$$\alpha=\int_{0}^{x}f_{\alpha}(t)t^{2k_{0}}(1-t)^{2k_{0}}\tan(\pi t)dt.$$
\end{lemma}
\begin{proof}
Fixed an $x\in[0,\frac{1}{2})$, define a functional on $C^{\infty}[0,x]$ by
\begin{equation}
T_{x,k_{0}}(f(t))=\int_{0}^{x}f(t)t^{2k_{0}}(1-t)^{2k_{0}}\tan(\pi t)dt,
\end{equation}
for $f(t)\in C^{\infty}[0,x]$.
Since $$\int_{0}^{x}t^{2k_{0}}(1-t)^{2k_{0}}\tan(\pi t)dt=K_{x,k_{0}}>0,$$
we have 
\begin{equation}
\begin{aligned}
|T_{x,k_{0}}(f(t))|&=\left|\int_{0}^{x}f(t)t^{2k_{0}}(1-t)^{2k_{0}}\tan(\pi t)dt\right|\\
&\leq K_{x,k_{0}}\sup_{t\in [0,x]} |f(t)|
\end{aligned}
\end{equation}
for $f(t)\in C^{\infty}[0,x]$, 
thus $T_{x,k_{0}}(f(t))$ is a bounded linear functional under the supremum metric.
Then fixed a small $\delta>0$, define a smooth function
\begin{equation}
f_{1}(t)=
\begin{cases}
M,&t\in\left[\frac{x}{2}-\delta,\frac{x}{2}+\delta\right]\\
0,&t\in\left[0, \frac{x}{2}-2\delta\right]\cup\left[\frac{x}{2}+2\delta, x\right].
\end{cases}
\end{equation}
Since
$$\inf_{t\in\left[\frac{x}{2}-\delta,\frac{x}{2}+\delta\right]}\left\{t^{2k_{0}}(1-t)^{2k_{0}}\tan(\pi t)\right\}=L_{x, k_{0}}>0,$$
we have 
\begin{equation} 
\begin{aligned}
T_{x,k_{0}}(f_{1}(t))&=\int_{0}^{x}f_{1}(t)t^{2k_{0}}(1-t)^{2k_{0}}\tan(\pi t)dt\\
&\geq 2M\delta L_{x,k_{0}}.
\end{aligned}
\end{equation}
Similarly, define a smooth function
\begin{equation} 
f_{2}(t)=
\begin{cases}
N,&t\in\left[\frac{x}{2}-\delta,\frac{x}{2}+\delta\right] \\
0,&t\in\left[0, \frac{x}{2}-2\delta\right]\cup\left[\frac{x}{2}+2\delta, x\right],
\end{cases}
\end{equation}
we have
\begin{equation}
\begin{aligned}
T_{x,k_{0}}(f_{2}(t))&=\int_{0}^{x}f_{2}(t)t^{2k_{0}}(1-t)^{2k_{0}}\tan(\pi t)dt\\
&\leq NK_{x,k_{0}}.
\end{aligned}
\end{equation}
For any $\alpha>0$, taking $N>0$ small enough such that 
\begin{equation}\label{T1}
NK_{x,k_{0}}<\alpha
\end{equation}
and taking $M>0$ big enough such that 
\begin{equation}\label{T2}
2M\delta L_{x,k_{0}}>\alpha .
\end{equation}
For any $0\leq\lambda\leq 1$, define
$$f^{(\lambda)}(t)=\lambda f_{1}(t)+(1-\lambda)f_{2}(t).$$
Obviously, it is in $C^{\infty}[0,x]$.
Thus for $\lambda\in [0,1]$,
$$g(\lambda)=T_{x,k_{0}}(f^{(\lambda)}(t))$$ becomes a smooth function from $[0,1]$ to $\mathbb{R}$.
Since by (\ref{T1}) and (\ref{T2}) we have
$$T_{x,k_{0}}(f_{2}(t)) < \alpha < T_{x,k_{0}}(f_{1}(t)),$$
from the intermediate value theorem, there exists an $\lambda\in(0,1)$ such that $f^{(\lambda)}(t)$ satisfying
$$\alpha=T_{x,k_{0}}(f^{(\lambda)}(t))=\int_{0}^{x}f^{(\lambda)}(t)t^{2k_{0}}(1-t)^{2k_{0}}\tan(\pi t)dt.$$
Denote $f_{\alpha}(t)=f^{(\lambda)}(t)$.
If $\alpha<0$, then $-\alpha>0$ and from the above result there exists a $f_{-\alpha}(t)\in C^{\infty}[0,x]$ such that
$$-\alpha=T_{x,k_{0}}(f_{-\alpha}(t)).$$
Let $f_{\alpha}=-f_{-\alpha}$, then 
$$T_{x,k_{0}}(f_{\alpha}(t))=-T_{x,k_{0}}(f_{-\alpha}(t))=\alpha$$
satisfying our requirement.
If $\alpha=0$, taking $f_{\alpha}(t)\equiv 0$ for $t\in[0,x]$, then we have
$$T_{x,k_{0}}(f_{\alpha}(t))=0=\alpha.$$
We thus obtain Lemma \ref{2.2}.
\end{proof}

Now we are at the position to prove the approximation property of the set $B$.
\begin{theorem}\label{Theorem 1}
Fixed an $x\in[0,\frac{1}{2})$. Let $k_{0}$ and $q$ be positive integers, let $\alpha$ be a real number. For any given integer $n\geq 3$, there exists an integer $r\leq n$ and a rational polynomial $P_{n}(t)$
with degree $4k_{0}+r$ depending on $\alpha,n,q,k_{0}$ and $x$:
$$P_{n}(t)=a_{2k_{0}}t^{2k_{0}}+\cdots+a_{4k_{0}+r}t^{4k_{0}+r},$$
satisfying 
$$\left|\alpha-\sum_{k=2k_{0}}^{4k_{0}+r}c_{k}\frac{\log\C_{k+1}(x)}{\pi}\right|\ll_{\alpha,q, k_{0}, x}\frac{1}{n^q},$$
where $$c_{k}=-a_{k}.$$
\end{theorem}
\begin{proof}
Let $n\geq 3$ be a given integer. For $f(t)\in C^{\infty}[0,x]$, by a classical theorem of Jackson (see \cite[p. 3749, (2.16)]{Alkan}), 
there exists a rational polynomial of degree $r\leq n$, whose coefficients depend only on $f,n,q$ and $x$:
$$s_{r}(t)=a_{0}+a_{1}t+\cdots + a_{r}t^{r},$$
such that
\begin{equation}\label{T3}
\|f-s_{r}\|=\textrm{sup}_{t\in[0,x]}|f(t)-s_{r}(t)|\ll_{f,q,x}\frac{1}{n^{q}}.
\end{equation}
And by Lemma \ref{2.2}, there is a function $f_{\alpha}(t)\in C^{\infty}[0,x]$ such that
$$\alpha=\int_{0}^{x}f_{\alpha}(t)t^{2k_{0}}(1-t)^{2k_{0}}\tan(\pi t)dt.$$
Let
\begin{equation}\label{T4}
\begin{aligned}
P_{n}(t)&=t^{2k_{0}}(1-t)^{2k_{0}}s_{r}(t)\\
&=a_{2k_{0}}t^{2k_0}+\cdots+a_{4k_{0}+r}t^{4k_{0}+r}.
\end{aligned}
\end{equation}
Then applying (\ref{T3}) to $f_{\alpha}(t)$, we get
\begin{equation}
\begin{aligned}
&\left|\alpha -\int_{0}^{x}P_{n}(t)\tan(\pi t)dt\right|\\
&\qquad=\left|\int_{0}^{x}f_{\alpha}(t)t^{2k_{0}}(1-t)^{2k_{0}}\tan(\pi t)-\int_{0}^{x}P_{n}(t)\tan(\pi t)dt\right|\\
&\qquad\leq\int_{0}^{x}|f_{\alpha}(t)-s_{r}(t)|t^{2k_{0}}(1-t)^{2k_{0}}\tan(\pi t)dt\\
&\qquad\ll_{\alpha,q,k_{0},x}\frac{1}{n^{q}}.
\end{aligned}
\end{equation}
Now we calculate the integral
$\int_{0}^{x}P_{n}(t)\tan(\pi t)dt$.
By Lemma \ref{2.1}, we have
\begin{equation}
\begin{aligned}
\int_{0}^{x}P_{n}(t)\tan(\pi t)dt&=\sum_{k=2k_{0}}^{4k_{0}+r}a_{k}\int_{0}^{x}t^{k}\tan(\pi t)dt\\
&=\sum_{k=2k_{0}}^{4k_{0}+r}\left(-\frac{a_{k}}{\pi}\right)\log\C_{k+1}(x).
\end{aligned}
\end{equation}
Combing with (\ref{T4}) we get
$$\left|\alpha-\sum_{k=2k_{0}}^{4k_{0}+r}c_{k}\frac{\log\C_{k+1}(x)}{\pi}\right|\ll_{\alpha,q, k_{0}, x}\frac{1}{n^q},$$
where $$c_{k}=-a_{k}.$$
\end{proof}

\section{The integral $\int_0^1 t^{r}\log\left(\cos\frac{\pi t}{4}\right)dt$}
 In \cite[Theorem 2.3]{HK2025}, we have calculated the integral $\int_0^x t^{r-2}\log\left(\cos\frac{t}{2}\right)dt$
 for  $r=2,3,4,\ldots$ as follows
\begin{equation}\label{T5}
\begin{aligned}
\int_0^{\frac\pi2} t^{r-2}\log\left(\cos\frac{t}{2}\right)dt 
&=-\frac{\log2}{r-1}\left(\frac\pi2\right)^{r-1}+(r-2)!\sin\left(\frac{r\pi}{2}\right)\zeta_E(r) \\
&\quad+\sum_{k=0}^{\left\lfloor \frac{r-2}{2}\right\rfloor }(-1)^k(2k)!\binom{r-2}{2k}
\left(\frac\pi2\right)^{r-2k-2} \\
&\quad\times\beta(2k+2) \\
&\quad+\sum_{k=1}^{\left\lceil \frac{r-2}2\right\rceil}\frac{(-1)^{k-1}(2k-1)!}{2^{2k+1}}\binom{r-2}{2k-1}
\left(\frac\pi2\right)^{r-2k-1} \\
&\quad\times\zeta_E(2k+1).
\end{aligned}
\end{equation}
Recall that $\zeta_{E}(s)$ and $\beta(s)$ are Dirichlet's eta and beta functions, respectively (see (\ref{A-zeta-1}) and (\ref{beta-def})).

Let 
$$D=\left\{\frac{\zeta_{E}(3)}{\pi^2},\frac{\zeta_{E}(5)}{\pi^4}, \ldots, \frac{\zeta_{E}(2k+1)}{\pi^{2k}},\ldots; \frac{\beta(4)}{\pi^3},\frac{\beta(6)}{\pi^5}, \ldots, \frac{\beta(2k+2)}{\pi^{2k+1}},\ldots\right\}$$
be the set of zeta and beta values.
In this section, we will show that the real numbers can be strongly approximated by
linear combinations of elements in $D$, where the coefficients are values of the derivatives of rational polynomials.
First, we state the following integral representation which is the $\log\left(\cos\frac{\pi t}{4}\right)$-counterpart of Lemma \ref{2.1}.
 \begin{lemma}\label{3.1}
For $r=0,1,2,\ldots$, we have
\begin{equation}
\begin{aligned}
\int_{0}^{1}t^{r}\log\left(\cos\frac{\pi t}{4}\right)&=-\frac{\log 2}{r+1}-r!\sin\left(\frac{r\pi}{2}\right) 2^{r+1}\frac{\zeta_{E}(r+2)}{\pi^{r+1}}\\
&\quad+\sum_{k=0}^{\lfloor\frac{r}{2}\rfloor}(-1)^{k}(2k)!\binom{r}{2k}2^{2k+1}\frac{\beta(2k+2)}{\pi^{2k+1}}\\
&\quad+\frac{1}{2}\sum_{k=1}^{\left\lceil \frac{r}{2}\right\rceil}(-1)^{k}(2k-1)!\binom{r}{2k-1}\frac{\zeta_{E}(2k+1)}{\pi^{2k}}.
\end{aligned}
\end{equation}
\end{lemma} 
\begin{proof}
By (\ref{T5}), for $r=2,3,4, \ldots$ we have
\begin{equation}
\begin{aligned}
\int_{0}^{1}t^{r-2}\log\left(\cos\frac{\pi t}{4}\right)dt&=\left(\frac{2}{\pi}\right)^{r-1}\int_{0}^{\frac{\pi}{2}}t^{r-2}\log\left(\cos\frac{t}{2}\right)dt\\
&=\left(\frac{2}{\pi}\right)^{r-1}\bigg(-\frac{\log 2}{r-1}\left(\frac{\pi}{2}\right)^{r-1}+(r-2)!\sin\left(\frac{r \pi}{2}\right)\zeta_{E}(r)\\
&\quad+\sum_{k=0}^{\lfloor\frac{r-2}{2}\rfloor}(-1)^{k}(2k)!\binom{r-2}{2k}\left(\frac{\pi}{2}\right)^{r-2k-2}\beta(2k+2)\\
&\quad+\sum_{k=1}^{\left\lceil \frac{r-2}2\right\rceil}\frac{(-1)^{k-1}(2k-1)!}{2^{2k+1}}\binom{r-2}{2k-1}\left(\frac{\pi}{2}\right)^{r-2k-1}\zeta_{E}(2k+1)\bigg)\\
&=-\frac{\log 2}{r-1}+(r-2)!\sin\left(\frac{r \pi}{2}\right)\left(\frac{2}{\pi}\right)^{r-1}\zeta_{E}(r)\\
&\quad+\sum_{k=0}^{\lfloor\frac{r-2}{2}\rfloor}(-1)^{k}(2k)!\binom{r-2}{2k}\left(\frac{2}{\pi}\right)^{2k+1}\beta(2k+2)\\
&\quad+\sum_{k=1}^{\left\lceil \frac{r-2}2\right\rceil}\frac{(-1)^{k-1}(2k-1)!}{2^{2k+1}}\binom{r-2}{2k-1}\left(\frac{2}{\pi}\right)^{2k}\zeta_{E}(2k+1),
\end{aligned}
\end{equation}
which is equivalent to
\begin{equation}
\begin{aligned}
\int_{0}^{1}t^{r}\log\left(\cos\frac{\pi t}{4}\right)&=-\frac{\log 2}{r+1}-r!\sin\left(\frac{r\pi}{2}\right) 2^{r+1}\frac{\zeta_{E}(r+2)}{\pi^{r+1}}\\
&\quad+\sum_{k=0}^{\lfloor\frac{r}{2}\rfloor}(-1)^{k}(2k)!\binom{r}{2k}2^{2k+1}\frac{\beta(2k+2)}{\pi^{2k+1}}\\
&\quad+\frac{1}{2}\sum_{k=1}^{\left\lceil \frac{r}{2}\right\rceil}(-1)^{k}(2k-1)!\binom{r}{2k-1}\frac{\zeta_{E}(2k+1)}{\pi^{2k}}
\end{aligned}
\end{equation}
for $r=0,1,2, \ldots$.
\end{proof}
The above lemma implies the following result.
\begin{lemma}\label{3.2}
For any polynomial $P(t)\in\mathbb{C}[t]$ with $P(0)=0$ and $\deg P=n$, we have
\begin{equation}
\begin{aligned}
\int_{0}^{1}P(t)\log \left(\cos \frac{\pi t}{4}\right)dt&=-\frac{P(1)}{r+1}\log 2
+\sum_{k=0}^{\lfloor\frac{n}{2}\rfloor}(-1)^{k}P^{(2k)}(1)2^{2k+1}\frac{\beta(2k+2)}{\pi^{2k+1}}\\
&\quad+\sum_{k=1}^{\left\lceil \frac{n}{2}\right\rceil}\left[\frac{1}{2}(-1)^{k}P^{(2k-1)}(1)+P^{(2k-1)}(0)2^{2k}\right]\\
&\quad\times\frac{\zeta_{E}(2k+1)}{\pi^{2k}}.
\end{aligned}
\end{equation}
\end{lemma}
Then for any $x\in [0,1),$ considering the following functional on $C^{\infty}[0,x]$:
\begin{equation}
C_{x,k_{0}}(f(t))=\int_{0}^{x}f(t)t^{2k_{0}}(1-t)^{2k_{0}}\log\left(\cos\frac{\pi t}{4}\right)dt,
\end{equation}
where $f(t)\in C^{\infty}[0,x]$.
For a rational polynomial $P_{n}(t)=t^{2k_{0}}(1-t)^{2k_{0}}s_{r}(t)$, notice that 
$P_{n}^{(2k)}(0)=P_{n}^{(2k)}(1)=0$ for $0\leq k\leq k_{0}-1$,
 the same reasoning as Theorem \ref{Theorem 1} implies the following result.
 \begin{theorem}\label{Theorem 2}
Let $k_{0}$ and $q$ be positive integers, let $\alpha$ be a real number. For any given integer $n\geq 3$, there exists an integer $r\leq n$ and a rational polynomial $P_{n}(t)=t^{2k_{0}}(1-t)^{2k_{0}}s_{r}(t)$
with degree $4k_{0}+r$ depending on $\alpha,n,q$ and $k_{0}$,
satisfying 
$$\left|\alpha-\sum_{k=k_{0}}^{2k_{0}+\lfloor \frac{r}{2}\rfloor}c_{k}^{\beta}\frac{\beta(2k+2)}{\pi^{2k+1}}-\sum_{k=k_{0}}^{2k_{0}+\lceil \frac{r}{2}\rceil}c_{k}^{\zeta_{E}}\frac{\zeta_{E}(2k+1)}{\pi^{2k}}\right|\ll_{\alpha,q, k_{0}}\frac{1}{n^q}$$
when $r\geq 2$, where $$c_{k}^{\beta}=(-1)^{k} P_{n}^{(2k)}(1) 2^{2k+1}$$
for all $k_{0}\leq k \leq 2k_{0}+\lfloor \frac{r}{2}\rfloor,$
and  $$c_{k}^{\zeta_{E}}=\frac{1}{2}(-1)^{k}P_{n}^{(2k-1)}(1)+P_{n}^{(2k-1)}(0)2^{2k}$$
all $k_{0}\leq k \leq 2k_{0}+\lceil\frac{r}{2}\rceil.$
\end{theorem} 
 
 \section{The integral $\int_0^x t^{r}\cot\left(\frac{\pi t}{2}\right)dt$}
In this section, we will consider the cotangent counterpart of section 2.
For any fixed $x\in[0,\frac{1}{2})$, let $$C=\left\{\frac{\log S_r(x)}{\pi}~~\bigg|~~r=1,2,3,\ldots\right\},$$
 we will show that  the real numbers can be approximated by
linear combinations of elements in $C$  with rational coefficients. First, we state the following integral representation which is the cotangent counterpart of Lemma \ref{2.1}.
\begin{lemma}\label{4.1}
For $0\leq x<1$ and $r=1,2,3,\ldots$, we have
\begin{equation}\label{T8}
\int_0^x t^{r}\cot\left(\frac{\pi t}{2}\right)dt=\frac{2^r}{\pi}\log S_{r}\left(\frac{x}{2}\right).
\end{equation}
\end{lemma}
\begin{proof}
For $0\leq x<1$ and $r=2,3,4,\ldots$, by \cite[Proposition 2]{KK05}, we have
\begin{equation}
\begin{aligned}
\frac{\log S_{r}\left(\frac{x}{2}\right)}{\pi}&=\int_{0}^{\frac{x}{2}}t^{r-1}\cot(\pi t)dt\\
&=\int_{0}^{x}\left(\frac{t}{2}\right)^{r-1}\cot\left(\frac{\pi t}{2}\right)d\left(\frac{t}{2}\right)\\
&=\frac{1}{2^r}\int_0^x t^{r-1}\cot\left(\frac{\pi t}{2}\right)dt,
\end{aligned}
\end{equation}
which is equivalent to (\ref{T8}).
\end{proof} 
Then fixed an $x\in [0,1),$ by considering the following functional on $C^{\infty}[0,x]$:
\begin{equation}
S_{x,k_{0}}(f(t))=\int_{0}^{x}f(t)t^{2k_{0}}(1-t)^{2k_{0}}\cot\left(\frac{\pi t}{2}\right)dt,
\end{equation}
for $f(t)\in C^{\infty}[0,x]$, the same reasoning as Theorem \ref{Theorem 1} implies the following result.
\begin{theorem}\label{Theorem 3}
Fixed an $x\in[0,1)$. Let $k_{0}$ and $q$ be positive integers, let $\alpha$ be a real number. For any given integer $n\geq 3$, there exists an integer $r\leq n$ and a rational polynomial $P_{n}(t)$
with degree $4k_{0}+r$ depending on $\alpha,n,q$ and $x$:
$$P_{n}(t)=a_{2k_{0}}t^{2k_{0}}+\cdots+a_{4k_{0}+r}t^{4k_{0}+r},$$
satisfying 
$$\left|\alpha-\sum_{k=2k_{0}}^{4k_{0}+r}c_{k}\frac{\log S_{k+1}\left(\frac{x}{2}\right)}{\pi}\right|\ll_{\alpha,q, k_{0}, x}\frac{1}{n^q},$$
where $$c_{k}=2^{k}a_{k}.$$
\end{theorem}
The following integral concerning  the case of $x=1$ was obtained by Lupu and Wu \cite{LW} recently. It can be viewed as a cotangent counterpart of Lemma \ref{3.2}.
\begin{lemma}[{\cite[Lemma 2.3]{LW}}]\label{4.2}
For any polynomial $P(t)\in\mathbb{C}[t]$ with $P(0)=0$ and $\deg P=n$, we have
\begin{equation}
\begin{aligned}
\int_{0}^{1}P(t)\cot\left(\frac{\pi t}{2}\right)dt&=\frac{2P(1)}{\pi}\log 2+2\sum_{k=1}^{\lfloor n/2\rfloor}(-1)^{k}P^{(2k)}(1)\cdot\left(1-\frac{1}{2^{2k}}\right)\\
&\quad\times\frac{\zeta(2k+1)}{\pi^{2k+1}}+2\sum_{k=1}^{\lfloor n/2\rfloor}(-1)^{k}P^{(2k)}(0)\frac{\zeta(2k+1)}{\pi^{2k+1}}.
\end{aligned}
\end{equation}
\end{lemma}
If considering the following functional on $C^{\infty}[0,1]$:
\begin{equation}
\begin{aligned}
S_{k_{0}}(f(t))&=\int_{0}^{1}f(t)t^{2k_{0}}(1-t)^{2k_{0}}\cot\left(\frac{\pi t}{2}\right)dt\\
&=\lim_{x\to 1}\int_{0}^{x}f(t)t^{2k_{0}}(1-t)^{2k_{0}}\cot\left(\frac{\pi t}{2}\right)dt,
\end{aligned}
\end{equation}
then Lemma \ref{4.2} and the same reasoning as Theorem \ref{Theorem 2} recover Lupu and Wu's result (see Theorem \ref{Theorem 4} above).

\end{document}